\documentclass[12pt]{article}%
\usepackage{amsmath}
\usepackage{amsfonts}
\usepackage{amssymb}
\usepackage[all]{xy}
\usepackage{graphicx}
\usepackage{hyperref}%
\providecommand{\U}[1]{\protect\rule{.1in}{.1in}}
\newtheorem{theorem}{Theorem}[section]

\newtheorem{definition}[theorem]{Definition}
\newtheorem{example}[theorem]{Example}

\newtheorem{lemma}[theorem]{Lemma}

\newtheorem{proposition}[theorem]{Proposition}
\newtheorem{remark}[theorem]{Remark}

\newenvironment{proof}[1][Proof]{\noindent\textbf{#1.} }{\ \rule{0.5em}{0.5em}}
\begin{document}

\title{Arens Regularity of the Fremlin Projective Tensor Product of Banach Lattice Algebras}
\author{J.Jaber, T. Marwa\\{\small Research Laboratory of Algebra, Topology, Arithmetic, and Order}\\{\small Department of Mathematics}\\{\small Faculty of Mathematical, Physical and Natural Sciences of Tunis}\\{\small Tunis-El Manar University, 2092-El Manar, Tunisia}}
\date{}
\maketitle

\begin{center}
\bigskip Abstract
\end{center}

\bigskip We study the Arens regularity of the Fremlin projective tensor
product of Banach lattice algebras. We first introduce and investigate
$L$-weakly almost periodic functionals, a class of functionals defined in
terms of $L$-weak compactness of their orbits. We obtain a characterization of
these functionals in terms of bounded disjoint sequences and use it to analyze
the Fremlin projective tensor product. In particular, we prove that, for
nontrivial Banach lattice algebras $A$ and $B$, Arens regularity of
$A\otimes_{\left\vert \pi\right\vert }B$ implies Arens regularity of both $A$
and $B$. The converse, however, fails: we construct two Arens regular Banach
lattice algebras whose Fremlin projective tensor product is not Arens regular.
We also establish several sufficient conditions for the Arens regularity of
$A\otimes_{\left\vert \pi\right\vert }B$, including conditions involving the
positive Schur property and order continuity of the norm on the space of
regular operators. Finally, we show that when both factors are Banach
$f$-algebras, their Fremlin projective tensor product is Arens regular.



\section{Introduction}

Let $A$ be a Banach algebra. It is well known that the second dual
$A^{\prime\prime}$ admits two natural extensions of the product, called the
first and second Arens products. The algebra $A$ is said to be Arens regular
if these two products coincide on $A^{\prime\prime}$. Arens regularity is
therefore a natural way of measuring whether the multiplication of a Banach
algebra admits a unique extension to its bidual.

A fundamental characterization, going back to the work of Arens and
Grothendieck, states that a Banach algebra $A$ is Arens regular if and only if
every continuous linear functional on $A$ is weakly almost periodic (wap).
More precisely, for $f\in A^{\prime}$, consider its orbit $O_{f}=\left\{
f.x:\text{ }\left\Vert x\right\Vert \leq1\right\}  $, where $f.x$ is defined
by $f.x:y\longmapsto f(xy)$. The functional $f$ is said to be weakly almost
periodic if $O_{f}$ is relatively weakly compact. Grothendieck's double limit
criterion ensures that the wap functionals, denoted by $\mathrm{wap}(A)$ form
the largest subspace of $A^{\prime}$ on which the two Arens products agree.
Consequently, Arens regularity is equivalent to the identity $\mathrm{wap}%
(A)=A^{\prime}$.

This characterization highlights a deep connection between the algebraic
structure of $A$ and the weak compactness properties of subsets of its dual.

In this paper we investigate Arens regularity in the setting of Banach lattice
algebras, with particular emphasis on the Fremlin projective tensor product.
If $A$ and $B$ are Banach lattice algebras, then their Fremlin projective
tensor product $A\otimes_{\left\vert \pi\right\vert }B$ is again a Banach
lattice algebra, with multiplication on elementary tensors given by
\[
(a\otimes b)(a^{\prime}\otimes b^{\prime})=aa^{\prime}\otimes bb^{\prime
}\text{.}%
\]
The Fremlin tensor product is therefore a natural construction that
simultaneously reflects the algebraic and lattice structures of the two
factors. A natural question is whether Arens regularity is preserved by this
tensor product. More precisely, one may ask whether $A$ and $B$ Arens regular
is equivalent to $A\otimes_{\left\vert \pi\right\vert }B$ is Arens regular.

Our first main result shows that, for nontrivial Banach lattice algebras,
Arens regularity of the tensor product forces Arens regularity of each factor.
The converse is false: we construct two Arens regular Banach lattice algebras
$A$ and $B$ for which $A\otimes_{\left\vert \pi\right\vert }B$ fails to be
Arens regular.

To prove these results, weak compactness can be refined using disjointness
techniques. Motivated by this, we introduce a new class of functionals, called
$L$- weakly almost periodic functionals ($L$-wap). If $A$ is a Banach lattice
algebra and $f\in A^{\prime}$, we call $f$ $L$-weakly almost periodic when its
orbit $O_{f}$ is $L$-weakly compact. We denote the corresponding space by
$\mathcal{L}$-\textrm{wap}$(A)$. Since $\mathcal{L}$-\textrm{wap}%
$(A)\subseteq$\textrm{wap}$(A)$, thus the equality $\mathcal{L}$%
-\textrm{wap}$(A)=A^{\prime}$ is a sufficient condition for Arens regularity.

Our first purpose is to develop a criterion for $L$-weakly almost periodicity
that is particularly suited to tensor products. We show that, for $f\in
A^{\prime}$, $L$-weak almost periodicity is equivalent to the vanishing of
$f.x_{n}$ in norm for every norm-bounded disjoint sequence $(x_{n})$ in $A$.
These characterizations allow the order structure to enter directly into the
analysis of the orbits of functionals. We then apply this framework to the
Fremlin projective tensor product. Finally, we establish several sufficient
conditions for the Arens regularity of the Fremlin projective tensor product.
In particular, we show that if both dual spaces $A^{\prime}$ and $B^{\prime}$
have the positive Schur property, then $A\otimes_{\left\vert \pi\right\vert
}B$ is Arens regular. We also obtain Arens regularity when the space of
regular operators $L^{r}(A,B)$ has an order continuous norm.

The paper is organized as follows. In Section 2, we recall the necessary
background on Banach lattice algebras and Arens regularity and develop the
theory of $L$-weakly almost periodic functionals. We establish
characterizations in terms of disjoint sequences and study their stability
properties, together with consequences arising from order continuity and the
positive Schur property. In Section 3, we apply these results to the Fremlin
projective tensor product. We characterize $L$-weakly almost periodic
functionals on the tensor product, prove that Arens regularity of the tensor
product implies Arens regularity of its factors, present the counterexample to
the converse implication, and derive several sufficient conditions for Arens regularity.

\section{Preliminaries}

\subsection{Banach lattice algebras}

We briefly recall some standard notions on Banach lattices and Banach lattice
algebras. For general background, we refer to \cite{Aliprantis}, \cite{Meyer}
and \cite{Wnuk}.

Let $X$ be a Banach lattice. A sequence $(x_{n})\subset$ $X$ is said to be
disjoint if $\left\vert x_{n}\right\vert \wedge\left\vert x_{m}\right\vert =0$
whenever $n\neq m$. A subset $D\subset X$ is called $L$-weakly compact if
every disjoint sequence in the solid hull
\[
S(D)=\left\{  x\in X:\left\vert x\right\vert \leq\left\vert y\right\vert
\text{ for some }y\in D\right\}
\]
satisfies $\left\Vert x_{n}\right\Vert \longrightarrow0$.

A Banach lattice $A$ equipped with a continuous product is called a Banach
lattice algebra if the product of positive elements is positive. If, in
addition, $A$ satisfies the following property
\[
f\wedge g=0\text{ and }0\leq h\in A\text{ imply }fh\wedge g=hf\wedge
g=0\text{,}%
\]
then $A$ is called a Banach $f$-algebra. Many classical algebras of functions
provide examples of semiprime Banach $f$-algebras.

We shall also use the space $\mathcal{L}^{r}(A,B)$ the linear span of positive
operators from $A$ into $B$, called the space of regular operators. When $B$
is Dedekind complete, $\mathcal{L}^{r}(A,B)$ is always a Banach lattice with
respect to the "r-norm" defined by
\[
\left\Vert T\right\Vert _{r}=\left\Vert \left\vert T\right\vert \right\Vert
=\sup\left\{  \left\Vert \left\vert T\right\vert (x)\right\Vert :x\geq0\text{
and }\left\Vert x\right\Vert \leq1\right\}  \text{.}%
\]

\subsection{Arens regularity}

let $A$ be a Banach algebra, and denote by $A^{\prime}$ and $A^{\prime\prime}$
its dual and bidual, respectively. For $x\in A$, $f\in A^{\prime}$ and $F,G\in
A^{\prime\prime}$, we define the following linear functionals on $A$.%

\begin{align*}
f.x  &  :y\longmapsto f(xy)\text{ \ \ \ \ \ \ \ \ \ \ }f\circ x:y\longmapsto
f(yx)\text{\ \ \ \ \ \ \ \ \ \ \ \ \ \ \ \ \ \ \ \ }\\
F.f  &  :y\longmapsto F(f.y)\text{ \ \ \ \ \ \ \ }F\circ f:y\longmapsto
F(f\circ y)\text{\ \ \ \ \ \ \ \ \ \ \ \ \ \ \ \ }%
\end{align*}

The first and the second Arens products on $A^{\prime\prime}$ are defined by
\begin{equation}
F\bullet G:f\longmapsto F(G.f)\text{ and }F\circ G:f\longmapsto G(F\circ
f)\text{ for all }F,G\in A^{\prime\prime}\text{.} \label{1}%
\end{equation}

These are known as the first and second Arens products, respectively. Each of
them makes $A^{\prime\prime}$ into a Banach algebra, and $A$ embeds
isometrically into $A^{\prime\prime}$ via the canonical homomorphism
$a\longrightarrow a^{\prime\prime}$, where $a^{\prime\prime}(f)=f(a)$ for all
$f\in A^{\prime}$. \ 

The algebra $A$ is said to be Arens regular if the two Arens products on its
second dual $A^{\prime\prime}$ coincide, that is $F\bullet G=F\circ G$ for all
$F,G\in A^{\prime\prime}$.

For $f\in A^{\prime}$, the orbit of $f$ is $O_{f}=\left\{  f.x:\left\Vert
x\right\Vert \leq1\right\}  $.

We say that $f$ is a weakly almost periodic, or wap, if $O_{f}$ is relatively
weakly compact subset in $A^{\prime}$. We denote by \textrm{$wap$}$(A)$ the
set of all such functionals. By Grothendieck's double limit criterion, $f\in
A^{\prime}$ belongs to \textrm{$wap$}$(A)$ if and only if, for every pair of
sequences $(x_{n})_{n}$ and $(y_{n})_{n}$ in the closed unit ball $B_{A}$ of
$A$, we have
\[
\lim_{n}\lim_{m}f(x_{n}y_{m})=\lim_{m}\lim_{n}f(x_{n}y_{m})
\]
provided that both limits exist.

The two Arens products on $A^{\prime\prime}$ coincide, if and only if
\textrm{$wap$}$(A)=A^{\prime}$. Thus $A$ is Arens regular if, and only if the
double the double limit condition above holds for all $f\in A^{\prime}$ (see
\cite[Theorem 1.4.11]{Palmer}).

\subsection{$L$-weakly almost periodic functionals}

Let $A$ be a Banach lattice algebra then its bidual $A^{\prime\prime}$,
equipped with either Arens product, is again a Banach lattice algebra. The
canonical embedding of $A$ into $A^{\prime\prime}$ preserves both the
algebraic and the lattice structures.

The notion of weakly almost periodicity is based on weak compacteness of the
orbit of \ a functional. To use the disjointness properties of Banach
lattices, we introduce the following stronger form of weak almost periodicity.

Let $A$ be a Banach lattice algebra and $f\in A^{\prime}$. Let $f\in
A^{\prime}$, We say that $f$ is $L$-weakly almost periodic or $L$-wap if
$O_{f}$ is an $L$-weakly compact subset of $A^{\prime}$. We write
$\mathcal{L}$-\textrm{$wap$}$\mathrm{(}A)$ for the set of all $L$-wap functionals.

Since, every $L$-weakly compact subset is weakly compact (see
\cite[Proposition 3.6.5]{Meyer}),
\[
\mathcal{L}\text{-}\mathrm{wap(}A)\subseteq\mathrm{wap(}A)\subseteq A^{\prime
}\text{.}%
\]
Consequently, $\mathcal{L}$-$\mathrm{wap(}A)=A^{\prime}$ implies that $A$ is
Arens regular.

The following characterization will be used repeatedly, it can be viewed as an
analogue to Grothendieck's criterion for weakly almost periodicity in Banach
spaces, offering a practical tool for verifying $L$-wap property.

\begin{proposition}
\label{P1}Let $A$ be a Banach lattice algebra and $f\in A^{\prime}$. The
following are equivalent.

\begin{enumerate}
\item[(1)] $f\in\mathcal{L}$-\textrm{$wap$}$\mathrm{(}A)$.

\item[(2)] For every norm bounded disjoint sequence $(x_{n})$, one has
$\left\Vert f\circ x_{n}\right\Vert \longrightarrow0$ in $A^{\prime}$.

\item[(3)] For every norm bounded disjoint sequence $(x_{n})_{n\in\mathbb{N}}$
and every positive sequence $(y_{n})_{n\in\mathbb{N}}$ in $B_{A}$, we have
\[
\lim\limits_{n\longrightarrow\infty}f(y_{n}x_{n})=0\text{.}%
\]

\end{enumerate}
\end{proposition}

\begin{proof}
(1)$\Leftrightarrow$(2) follows from the standard characterization of
$L$-weakly compact subsets of a Banach lattice, see \cite[Theorem
5.63]{Aliprantis} or \cite[Proposition 3.6.3]{Meyer}.

(2)$\Longrightarrow$(3) This implication is straightforward.

(3)$\Longrightarrow$(2) Suppose, for contradiction, that there exists a
disjoint sequence $(x_{n})_{n\in\mathbb{N}}\subset$ $B_{A}$ such that
$(\left\Vert f\circ x_{n}\right\Vert )_{n\in\mathbb{N}}$ does not converges to
zero. Passing to a subsequence if necessary, we may assume that there exists
some $\varepsilon>0$ such that $\left\Vert f\circ x_{n}\right\Vert
\geq\varepsilon$ for all $n\in\mathbb{N}$. Then, for each $n\in\mathbb{N}$, we
can choose $0\leq y_{n}\in B_{A}$ such that $\left\vert f\circ x_{n}%
\right\vert (y_{n})\geq\dfrac{\varepsilon}{2}.$ Using the classical
Kantorovich formula
\[
\frac{\varepsilon}{2}\leq\left\vert f\circ x_{n}\right\vert (y_{n}%
)=\sup\left\{  \left\vert f\circ x_{n}(z)\right\vert :0\leq z\leq
y_{n}\right\}  \text{ for all }n\in\mathbb{N}\text{,}%
\]
there exists $0\leq z_{n}\leq y_{n}$ such that $\left\vert f\circ x_{n}%
(z_{n})\right\vert =\left\vert f(z_{n}x_{n})\right\vert \geq\dfrac
{\varepsilon}{4}$. This contradicts (3).
\end{proof}

\begin{example}
\label{E}\ \ \ \ \ \ \ \ \ \ \ \ \ \ \ \ \ \ \ \ \ \ \ \ \ \ \ \ \ \ \ \ \ \ \ \ \ 

\begin{enumerate}
\item Let $A=\ell^{1}(\mathbb{Z)}$ be the Banach lattice algebra with
pointwise multiplication
\[
x.y(n)=x_{n}y_{n}\text{ for all }x,y\in A\text{ and }n\in\mathbb{N}\text{.}%
\]
Then, $\mathcal{L}$-\textrm{$wap$}$(A)=c_{0}$.

If $\ell^{1}(\mathbb{Z)}$ is furnished with convolution product defined by
\[
(x.y)(n)=%
{\displaystyle\sum\limits_{k\in\mathbb{Z}}}
x(k)y(n-k)\text{ for all }x,y\in A\text{ and }n\in\mathbb{Z}\text{,}%
\]
Then $\mathcal{L}$-\textrm{$wap$}$(A)=\left\{  0\right\}  $.

\item $(\ell^{1}$-direct sum of Banach lattice algebras).

Let $(A_{n})_{n}$ be a sequence of Banach lattice algebras and consider their
\ $\ell^{1}$-direct sum
\[
A=\bigoplus\limits_{n=1}^{\infty}A_{n}=\left\{  x=(x_{n})_{n}:x_{n}\in
A_{n}\text{ and }\left\Vert x\right\Vert =\sum\limits_{n=1}^{\infty}\left\Vert
x_{n}\right\Vert <+\infty\right\}  \text{,}%
\]
With coordinatewise product
\[
x.y=(x_{1}y_{1},x_{2}y_{2},..)\text{ for all }x=(x_{n})\text{ }y=(y_{n})\in
A\text{.}%
\]
Then $A$ is a Banach lattice algebra. Moreover, the dual space $A^{\prime}$ is
given by
\[
A^{\prime}=\bigoplus\limits_{n=1}^{\infty}A_{n}=\left\{  f=(f_{n}):f_{n}\in
A_{n}^{\prime}\text{ and }\left\Vert f\right\Vert _{\infty}=\sup\left\Vert
f_{n}\right\Vert <+\infty\right\}  \text{.}%
\]

For every $f=(f_{n})\in A^{\prime}$ and $x=(x_{n})$, $f.x=(f_{1}.x_{1}%
,f_{2}.x_{2},...)$. So, the orbit decomposes coordinatewise $O_{f}=O_{f_{1}%
}\times O_{f_{2}}\times..$. Consequently,
\begin{align*}
\mathrm{wap}(A_{1}\oplus A_{2}\oplus...)_{1}  &  =\mathrm{wap}(A_{1}%
)\oplus\mathrm{wap}(A_{2})\oplus...\\
\mathcal{L}-\mathrm{wap}(A_{1}\oplus A_{2}\oplus...)_{1}  &  =\mathcal{L}%
-\mathrm{wap}(A_{1})\oplus\mathcal{L}-\mathrm{wap}(A_{2})\oplus...
\end{align*}

\end{enumerate}
\end{example}

It is well known that the set \textrm{$wap$}$(A)$ of weakly almost periodic
functionals on a Banach algebra is a closed vector subspace of the dual space
$A^{\prime}$. It is easy to see that the set $\mathcal{L}$-\textrm{$wap$}$(A)$
is a closed vector subspace of $A^{\prime}$. Moreover, if $0\leq
f\in\mathcal{L}$-\textrm{$wap$}$(A)$, then the ideal $I_{f}$ generated by $f$
that is
\[
I_{f}=\left\{  g\in A^{\prime}:\left\vert g\right\vert \leq\lambda f\text{ for
some }\lambda\right\}
\]
is contained in $\mathcal{L}$-\textrm{$wap$}$(A)$. In general, the space
$\mathcal{L}$-\textrm{$wap$} $(A)$ is not a vector sublattice of $A^{\prime}$,
as the following example shows.

\begin{example}
\bigskip Let $A=\ell^{1}(\mathbb{N)}$ with the multiplication defined by%
\[
x.y=(\sum\limits_{k=0}^{\infty}x_{k})(\sum\limits_{k=0}^{\infty}y_{k})v\text{
where }v=\frac{1}{2}(1,1,0,0,..)\text{.}%
\]
It is easy to check that $\ell^{1}$is a commutative Banach lattice algebra.
For $f\in A^{\prime}=\ell^{\infty}$, we have $f\circ x:y\longmapsto
(\sum\limits_{k=0}^{\infty}x_{k})(\sum\limits_{k=0}^{\infty}y_{k})f(v)$ for
all $f\in\ell^{\infty}$ and $x\in A$. Let $f=(1,-1,0,0,...)\in\ell^{\infty}$.
Then $f\in$ $\mathcal{L}$-\textrm{$wap$}$(A)$, since $O_{f}=\left\{
0\right\}  $, however $\left\vert f\right\vert =(1,1,0,0,..)\notin\mathcal{L}%
$-\textrm{$wap$}$(A)$, because, for the standard basis of $\ell^{1}$,
$\left\Vert \left\vert f\right\vert \circ e_{n}\right\Vert _{\infty
}=1\nrightarrow0$.
\end{example}

The main result of this section provides the connection between the order
structure of the dual space and the behavior of $L$-wap and wap functionals.

Recall that a Banach lattice $A$ is said to have the positive Schur property
(PSP) whenever $0\leq x_{n}\longrightarrow0$ weakly implies $\left\Vert
x_{n}\right\Vert \longrightarrow0$.

\begin{theorem}
\label{Main}Let $A$ be a Banach lattice.

\begin{enumerate}
\item If the dual space $A^{\prime}$ has an order continuous norm, then for
every Banach lattice algebra product $p$ on $A$, \textrm{$wap$}$\mathrm{(}%
A)=A^{\prime}$.

In particular, if $A^{\prime}$ has an order continuous norm then, $A$ equipped
with any product, is Arens regular.

\item If the dual space $A^{\prime}$ has the PSP then, for every Banach
lattice algebra product $p$ on $A$,
\[
\mathcal{L}\text{-}\mathrm{wap(}A)=\mathrm{wap(}A)=A^{\prime}.
\]

\item If for every Banach lattice algebra product $p$ on $A$, $\mathcal{L}%
$-\textrm{$wap$}$\mathrm{(}A)=A^{\prime}$ then $A^{\prime}$ has an order
continuous norm.
\end{enumerate}
\end{theorem}

\begin{proof}
\ \ \ 

\begin{enumerate}
\item Assume that $A^{\prime}$ has an order continuous norm. Then $A^{\prime}$
is weakly sequentially complete. Let $0\leq f\in A^{\prime}$ and consider a
sequence $(g_{n})_{n\in\mathbb{N}}$\ in the orbit $O_{f}$\ of $f$. Then there
exists a sequence $(x_{n})_{n\in\mathbb{N}}$ with $\left\Vert x_{n}\right\Vert
\leq1$ such that $g_{n}=f.x_{n}$. Since $A^{\prime}$ has order continuous
norm, it follows from \cite[Theorem 5.28]{Aliprantis} that $(g_{n}%
)_{n\in\mathbb{N}}$ admits a weak Cauchy subsequence. Consequently,
$(g_{n})_{n\in\mathbb{N}}$ has a weakly convergent subsequence. By
Eberlein-Smulian Theorem, $O_{f}$ is relatively weakly compact. Therefore,
$f\in$\textrm{$wap$}$\mathrm{(}A)$.

\item Assume that $A^{\prime}$ has the PSP then $A^{\prime}$ has an order
continuous norm. Let $p$ be a Banach lattice algebra product on $A$, and let
$0\leq f\in A^{\prime}$. By the first point the orbit $O_{f}$ is actually
relatively weakly compact subset. According to \cite[Theorem 4.34]%
{Aliprantis}, if $(g_{n})\subset S(O_{f})$ is a disjoint sequence then it
converges weakly to zero and, by the PSP, it converges to zero in norm. Thus
$f\in\mathcal{L}$-\textrm{$wap$}$\mathrm{(}A)$.

\item Assume that $\mathcal{L}$-\textrm{$wap$}$\mathrm{(}A)=A^{\prime}$ for
every Banach lattice algebra product $p$ on $A$. To show that $A^{\prime}$ has
order continuous norm, it suffices to prove that every norm bounded disjoint
sequence in $A$ converges weakly to zero. Let $(x_{n})_{n\in\mathbb{N}}$ be a
disjoint sequence in $B_{A}$. Consider a positive functional $0\leq\varphi\in
A^{\prime}$ with $\left\Vert \varphi\right\Vert =1$ and define the product
$.:A\times A\longrightarrow A$ by $x.y=\varphi(y)x$. Then $A$, equipped with
this product becomes a Banach lattice algebra. Let $f\in A^{\prime}$, for
every $x\in A$, the action of $f\circ x$ is given by $f\circ x=f(x)\varphi$.
Since $f\in\mathcal{L}$-\textrm{$wap$}$\mathrm{(}A)$, $\left\Vert f\circ
x_{n}\right\Vert =\left\vert f(x_{n})\right\vert \longrightarrow0$. Therefore,
$x_{n}\longrightarrow0$ weakly in $A$. This shows that $A^{\prime}$ has an
order continuous norm.
\end{enumerate}
\end{proof}

The first point of this theorem appeared in \cite[Theorem 1]{Buskes} in a more
general setting, namely for the Arens regularity of bilinear maps.

It is well known that the bidual of a Banach $f$-algebra is again a Banach
$f$-algebra equipped with either the first or the second Arens product (see
\cite[Theorem3.5]{BH}), in particular, it is commutative. It follows that
every Banach $f$-algebra is Arens regular.

At the end of this section, we give a direct proof of the Arens regularity of
Banach $f$-algebras using Theorem \ref{Main}, which avoids the classical and
non trivial fact that the bidual is an $f$-algebra.

\begin{proposition}
Let $A$ be a Banach $f$-algebra. Then \textrm{$wap$}$(A)=A^{\prime}$ , and
consequently $A$ is Arens regular.
\end{proposition}

\begin{proof}
Assume first that $A$ has a unit element. Then its dual space $A^{\prime}$ has
an order continuous norm (see \cite[Theorem 3.3]{depagter}). By Theorem
\ref{Main}, this implies that \textrm{$wap$}$(A)=A^{\prime}$.

If $A$ is semiprime $f$-algebra. Then $A$ is continuously embedded in the
unital Banach algebra $\mathcal{Z}(A)$, via the mapping $x\longmapsto\pi_{x}$
where $\pi_{x}:y\longmapsto xy$. Since $\mathcal{Z}(A)$ is Arens regular, $A$
is Arens regular, because every subalgebra of an Arens regular algebra is
Arens regular (\cite[Corollary 1.4.12]{Palmer}).

In the general case, let $N$ denote the set of all nilpotent elements in $A$
and let
\[
M=N^{\perp}=\left\{  x\in A:x\perp n\text{ for all }n\in N\right\}
\]
be the disjoint complement of $N$. Then $M$ is a semiprime $f$-algebra, and
moreover it satisfies $xy\in M$ for all $x,y\in A$ (see \cite[Proposition
10.2]{BD}). Let $f\in A^{\prime}$ and denote by $g=f_{/M}$ its restriction to
$M$. Let $(x_{n})$ and $(y_{n})$ be two sequences in the closed unit
ball$\ B_{A}$, and assume that the iterated limits
\[
\lim_{n}\lim_{m}f(x_{n}y_{m})=\alpha\text{ and }\lim_{m}\lim_{n}f(x_{n}%
y_{m})=\beta\text{ }%
\]
exist. Since $x_{n}y_{m}\in M$, we have $f(x_{n}y_{m})=g(x_{n}y_{m})$ for all
$n,m\in\mathbb{N}$. By the semiprime case, $g\in\mathrm{wap}(M)$, hence
$\alpha=\beta$, so $f\in\mathrm{wap}(A)$. Consequently, $A$ is Arens regular.
\end{proof}

\section{Arens regularity of the Fremlin projective tensor product.}

In this section, we investigate the Arens regularity of the Fremlin projective
tensor products of Banach lattice algebras.

Let $A$ and $B$ be Banach lattices. Fremlin in \cite{Fremlin} defined the
positive projective norm $\left\Vert .\right\Vert _{\left\vert \pi\right\vert
}$ on the algebraic tensor product $A\otimes B$ by
\[
\left\Vert u\right\Vert _{\left\vert \pi\right\vert }=\sup\left\{  \left\vert
\varphi(x,y)\right\vert :u=%
{\textstyle\sum\limits_{k=1}^{n}}
x_{k}\otimes y_{k},\text{ }\varphi\in L\right\}  \text{,}%
\]

where $L$ is the set of all positive bilinear maps from $A\times B$ into all
Banach lattices $G$ with norm $1$. The completion of $A\otimes B$ with respect
to this norm is denoted by $A\otimes_{\left\vert \pi\right\vert }B$ and is
called the Fremlin projective tensor product of $A$ and $B$. Fremlin
(\cite[Theorem 1E (vi)]{Fremlin}) showed that the Riesz tensor product
$A\overline{\otimes}B$ is norm-dense in $A\otimes_{\left\vert \pi\right\vert
}B$ and the positive projective norm, when considered on $A\overline{\otimes
}B$, admits the following equivalent norm%
\[
\left\Vert u\right\Vert _{\left\vert \pi\right\vert }=\inf\left\{
{\textstyle\sum\limits_{k=1}^{n}}
\left\Vert x_{k}\right\Vert \left\Vert y_{k}\right\Vert :x_{k}\in A^{+}\text{,
}y_{k}\in B^{+\text{ }}\text{ and }\left\vert u\right\vert \leq%
{\textstyle\sum\limits_{k=1}^{n}}
x_{k}y_{k},\text{ }\right\}
\]

\bigskip The Fremlin projective tensor product $A\otimes_{\left\vert
\pi\right\vert }B$ satisfies the following universal property

$(UP):$ For every Banach lattice $G$ and for every continuous bilinear map
$\varphi:A\times B\longrightarrow G$ there exists a unique continuous linear
map $T:A\otimes_{\left\vert \pi\right\vert }B\longrightarrow G$ such that
\begin{align*}
T(x\otimes y)  &  =\varphi(x,y)\text{, }x\in A,y\in B\text{ and }\\
\left\Vert T\right\Vert  &  =\left\Vert \varphi\right\Vert \text{.}%
\end{align*}

Now let $A$ and $B$ be Banach lattice algebras. Then $A\otimes_{\left\vert
\pi\right\vert }B$ admits a natural algebra structure, under which it becomes
a Banach lattice algebra (see \cite{Jaber}). The multiplication on
$A\otimes_{\left\vert \pi\right\vert }B$ is defined by linear extension from
elementary tensors via
\[
(a\otimes b).(a^{\prime}\otimes b^{\prime})=(aa^{\prime}\otimes bb^{\prime
})\text{.}%
\]

Our first aim is to obtain a useful description of the $L$-wap functionals on
this tensor product. For a subset $X$ of a vector lattice, let $CSCH(X)$
denote its closed solid convex hull. We first recall the following standard fact.

\begin{lemma}
\label{le}The closed solid convex hull $CSCH(X)$ of an $L$-weakly compact set
$X$ in a Banach lattice $E$ is $L$-weakly compact.
\end{lemma}

\begin{proof}
This follows from the stability of $L$-weak compactness under solid convex
hulls and norm closure.
\end{proof}

We shall use the following description of the unit ball.

Let $A$ and $B$ be two Banach lattices. The following result appears to be
known, however, we were unable to find an explicit reference.

\begin{proposition}
\label{P3}For Banach lattices $A$ and $B$,
\[
B_{A\otimes_{\left\vert \pi\right\vert }B}=CSCH(B_{A^{+}}\otimes B_{B^{+}}),
\]
where
\[
B_{A^{+}}\otimes B_{B^{+}}=\left\{  x\otimes y:0\leq x\in B_{A}\text{, }0\leq
y\in B_{B}\right\}
\]
\qquad
\end{proposition}

\begin{proof}
It is clear that $B_{A^{+}}\otimes B_{B^{+}}\subseteq B_{A\otimes_{\left\vert
\pi\right\vert }B}$ and hence $CSCH(B_{A}\otimes B_{B})\subseteq
B_{A\otimes_{\left\vert \pi\right\vert }B}$. For the reverse inclusion, it is
enough to show that every $u$ $\in A\overline{\otimes}B$ satisfying
$\left\Vert u\right\Vert \leq1$ belongs to $CSCH(B_{A^{+}}\otimes B_{B^{+}})$.
First, consider $u\in A\overline{\otimes}B$ with $\left\Vert u\right\Vert <1$.
The definition of the positive projective norm gives positive vectors
$x_{k}\in A$ and $y_{k}\in B$ such that $\left\vert u\right\vert \leq%
{\displaystyle\sum\limits_{k=1}^{n}}
x_{k}\otimes y_{k}$ and $%
{\displaystyle\sum\limits_{k=1}^{n}}
\left\Vert x_{k}\right\Vert \left\Vert y_{k}\right\Vert <1$. Let
$x_{i}^{\prime}=\dfrac{x_{i}}{\left\Vert x_{i}\right\Vert }$, $y_{i}^{\prime
}=\dfrac{y_{i}}{\left\Vert y_{i}\right\Vert }$ and $\lambda_{i}=\left\Vert
x_{i}\right\Vert \left\Vert y_{i}\right\Vert $. Then $\left\vert u\right\vert
\leq%
{\displaystyle\sum\limits_{k=1}^{n}}
\lambda_{k}x_{k}^{\prime}\otimes y_{k}^{\prime}$ with $\lambda_{k}\geq0$ and $%
{\displaystyle\sum\limits_{k=1}^{n}}
\lambda_{k}<1$. Therefore $u$ belongs to the solid convex hull $CH(B_{A^{+}%
}\otimes B_{B^{+}})$ of $B_{A^{+}}\otimes B_{B^{+}}$. Now, if $\left\Vert
u\right\Vert =1$, the sequence $((1-\frac{1}{n})u)$ is contained in
$SCH(B_{A^{+}}\otimes B_{B^{+}})$ and converges to $u$, hence $u\in
CSCH(B_{A^{+}}\otimes B_{B^{+}})$. It follows that the closed unit ball of
$A\overline{\otimes}B$ is contained in $CSCH(B_{A^{+}}\otimes B_{B^{+}})$.
\end{proof}

To characterize almost periodicity of functionals on the Fremlin projective
tensor product, we need the following notion, which appears in
\cite[Definition 5.3]{m-weak}

\begin{definition}
Two vectors $(x,y)$ and $(x^{\prime},y^{\prime})$ in $A\times B$ are said to
be sporadically disjoint if $x\perp x^{\prime}$ or $y\perp y^{\prime}$.

A sequence $(x_{n},y_{n})$ in $A\times B$ is called sporadically disjoint
sequence if for $n\neq m$ the pairs $(x_{n},y_{n})$ and $(x_{m},y_{m})$ are
sporadically disjoint .
\end{definition}

The following lemma is well known, see \cite[Lemma 5.5]{m-weak}. We include
the proof for completeness in this simplified setting.

\begin{lemma}
Let $x\otimes y$ and $x^{\prime}\otimes y^{\prime}\in A\otimes_{\left\vert
\pi\right\vert }B$. Then $x\otimes y\perp x^{\prime}\otimes y^{\prime}$ if and
only if $(x,y)$ and $(x^{\prime},y^{\prime})$ are sporadically disjoint.
\end{lemma}

\begin{proof}
Suppose that $x\perp x^{\prime}\in A$. Then
\begin{align*}
\left\vert x\right\vert \otimes\left\vert y\right\vert \wedge\left\vert
x^{\prime}\right\vert \otimes\left\vert y\right\vert  &  \leq\left\vert
x\right\vert \otimes(\left\vert y\right\vert \vee\left\vert y^{\prime
}\right\vert )\wedge\left\vert x^{\prime}\right\vert \otimes(\left\vert
y\right\vert \vee\left\vert y^{\prime}\right\vert )\\
&  =\left\vert x\right\vert \wedge\left\vert x^{\prime}\right\vert
\otimes(\left\vert y\right\vert \vee\left\vert y^{\prime}\right\vert
)=0\text{,}%
\end{align*}
and therefore, $x\otimes y\perp x^{\prime}\otimes y^{\prime}$.

Conversely, assume that $\left\vert x\right\vert \otimes\left\vert
y\right\vert \wedge\left\vert x^{\prime}\right\vert \otimes\left\vert
y^{\prime}\right\vert =0$. From the inequality
\[
\left\vert x\right\vert \wedge\left\vert x^{\prime}\right\vert \otimes
\left\vert y\right\vert \wedge\left\vert y^{\prime}\right\vert \leq\left\vert
x\right\vert \otimes\left\vert y\right\vert \wedge\left\vert x^{\prime
}\right\vert \otimes\left\vert y^{\prime}\right\vert =0\text{,}%
\]
it follows that $\left\vert x\right\vert \wedge\left\vert x^{\prime
}\right\vert \otimes\left\vert y\right\vert \wedge\left\vert y^{\prime
}\right\vert =0$ . Thus, $\left\vert x\right\vert \wedge\left\vert x^{\prime
}\right\vert =0$ or $\left\vert y\right\vert \wedge\left\vert y^{\prime
}\right\vert =0$.
\end{proof}

To study the orbit of a functional $f$ on $A\otimes_{\left\vert \pi\right\vert
}B$, we face the difficulty of analyzing the complex set $O_{f}=\left\{
f.u\text{: }\left\Vert u\right\Vert _{\pi}\leq1\right\}  $. The next result
shows that it is sufficient to restrict attention to decomposable tensors
$x\otimes y\in B_{A^{+}}\otimes B_{B^{+}}$ instead of arbitrary elements $u$
in the closed unit ball $B_{A\otimes_{\left\vert \pi\right\vert }B}$.

\begin{proposition}
\label{KK}Let $f\in(A\otimes_{\left\vert \pi\right\vert }B)^{\prime}$. Then
the following assertions are equivalent:

\begin{description}
\item[(1)] $f\in\mathcal{L}$-\textrm{wap}$(A\otimes_{\left\vert \pi\right\vert
}B)$

\item[(2)] The set $W_{f}=\left\{  f.x\otimes y:0\leq x\in B_{A}\text{,}0\leq
y\in B_{B}\right\}  $ is $L$-weakly compact.

\item[(3)] For every positive sporadically disjoint sequence $(x_{n}%
,y_{n})\subset B_{A\times B}$ and every sequence $(x_{n}^{\prime}%
,y_{n}^{\prime})\subseteq B_{A\times B}$
\[
f(x_{n}^{\prime}x_{n}\otimes y_{n}^{\prime}y_{n})\longrightarrow0
\]

\end{description}
\end{proposition}

\begin{proof}
(1)$\Longrightarrow$(2) follows from the fact that $W_{f}\subseteq O_{f}$ and
$O_{f}$ is $L$-weakly compact.

(2)$\Longrightarrow$(1) Assume that $W_{f}$ is $L$-weakly compact. Using the
linearity and the continuity of the operator $T_{f}:A\otimes_{\left\vert
\pi\right\vert }B\longrightarrow(A\otimes_{\left\vert \pi\right\vert
}B)^{\prime}$ defined by $T_{f}(u)=f.u$ and  proposition \ref{P3}, we obtain
\[
O_{f}=T_{f}(B_{A\otimes_{\left\vert \pi\right\vert }B})=T_{f}(CSCH(B_{A^{+}%
}\otimes B_{B^{+}}))\subseteq CSCH(W_{f})
\]
By Lemma \ref{le}, $CSCH(W_{f})$ is $L$-weakly compact, hence $O_{f}$ is
$L$-weakly compact, and therefore $f\in\mathcal{L}$-\textrm{wap}%
$(A\otimes_{\left\vert \pi\right\vert }B)$.

(2)$\Leftrightarrow$(3) The set $W_{f}$ is $L$-weakly compact if and only if,
for all disjoint sequence $(g_{n})\subseteq S(W_{f})$ converges uniformly to
zero on $B_{A\otimes_{\left\vert \pi\right\vert }B}$, that is,
\[
\left\Vert g_{n}\right\Vert =\sup\limits_{u\text{ }\in B_{A\otimes_{\left\vert
\pi\right\vert }B}}\left\vert g_{n}(u)\right\vert \longrightarrow0
\]
Using the universal property of the Fremlin projective tensor product
\[
\left\Vert g_{n}\right\Vert =\sup\limits_{x\in B_{A^{+}},y\in B_{B^{+}}%
}\left\vert g_{n}(x\otimes y)\right\vert \longrightarrow0
\]
for every disjoint sequence $(g_{n})_{n}\subseteq S(W_{f})$. Thus, $(g_{n})$
converges uniformly to zero on $B_{A^{+}}\otimes B_{B^{+}}$. By \cite[theorem
5.63]{Aliprantis}, this is equivalent to the fact that for every disjoint
sequence $(x_{n}\otimes y_{n})\in B_{A^{+}}\otimes B_{B^{+}}$, we have
\[
\sup_{h\in W_{f}}\left\vert h(x_{n}\otimes y_{n})\right\vert =\sup_{x\in
B_{A},\text{ }y\in B_{B^{+}}}\left\vert f.x\otimes y(x_{n}\otimes
y_{n})\right\vert \longrightarrow0
\]
Thus,$W_{f}$ is $L$-weakly compact if and only if for every positive
sporadically disjoint $(x_{n},y_{n})\subset B_{A\times B}$ and every sequence
$(x_{n}^{\prime},y_{n}^{\prime})\subseteq B_{A\times B}$
\[
f(x_{n}^{\prime}x_{n}\otimes y_{n}^{\prime}y_{n})\longrightarrow0\text{.}%
\]
This completes the proof.
\end{proof}

Let $A$ be a Banach algebra. We denote by $A^{2}=\left\{  ab:a,b\in A\right\}
$. If $A^{2}=0$ , we say that $A$ is a trivial Banach algebra . It is clear
that the Fremlin projective tensor product algebra $A\otimes_{\left\vert
\pi\right\vert }B$ is trivial if and only if at least one of the algebras $A$
or $B$ is trivial. Therefore, throughout the remainder of this section, we
assume that the algebras under consideration are nontrivial.

Let $f\in A^{\prime}$ and $g\in B^{\prime}$ we define the linear functional
$f\otimes g\in(A\otimes_{\left\vert \pi\right\vert }B)^{\prime}$ by
\[
f\otimes g(x\otimes y)=f(x)g(y)\text{, \ }x\in A\text{ and }y\in B.\
\]
As an application of \ proposition \ref{KK}, we obtain the following result.

\begin{proposition}
\label{PPP}

\begin{enumerate}
\item $f\in\mathcal{L}$-$\mathrm{wap}(A)$, $g\in\mathcal{L}$-$\mathrm{wap}%
(B)\Longrightarrow f\otimes g\in\mathcal{L}$-$\mathrm{wap}(A\otimes
_{\left\vert \pi\right\vert }B)$.

\item $f\otimes g\in\mathcal{L}$-$\mathrm{wap}(A\otimes_{\left\vert
\pi\right\vert }B)$, $g_{/B^{2}}\neq0$ $\Longrightarrow$ $f\in\mathcal{L}%
$-$\mathrm{wap}(A)$.

\item $f\otimes g\in\mathrm{wap}(A\otimes_{\left\vert \pi\right\vert }B)$,
$g_{/B^{2}}\neq0\Longrightarrow$ $f\in\mathrm{wap}(A)$
\end{enumerate}
\end{proposition}

\begin{proof}
\ \ \ \ \ \ \ \ \ 

\begin{enumerate}
\item Let $(x_{n},y_{n})$ be a positive sporadically disjoint sequence in
$B_{A}\times B_{B}$, and let $(x_{n}^{\prime},y_{n}^{\prime})$ be \ a sequence
in$\ B_{A}\times B_{B}$. We must show $\alpha_{n}=f(x_{n}^{\prime}%
x_{n})g(y_{n}^{\prime}y_{n})\longrightarrow0$. Suppose, to the contrary, that
this is not the case. Then, after passing to a subsequence, there exists
$\varepsilon>0$ such that~$\left\vert f(x_{n}^{\prime}x_{n})g(y_{n}^{\prime
}y_{n})\right\vert \geq\varepsilon$, for each $n$. For each $n\neq m$, either
$x_{n}\perp x_{m}$ or $y_{n}\perp y_{m}$. By Ramsey's theorem, there is an
infinite subsequence $(n_{k})$ such that either $(x_{n_{k}})$ is disjoint in
$A$ or $(y_{n_{k}})$ is disjoint in $B$. In the firt case, since
$f\in\mathcal{L}$-$\mathrm{wap}(A)$, $f(x_{n_{k}}^{\prime}x_{n_{k}%
})\longrightarrow0$. In the second case, since $g\in\mathcal{L}$%
-$\mathrm{wap}(B)$, $g(y_{n_{k}}^{\prime}y_{n_{k}})\longrightarrow0$. In
either case, $\ f(x_{n_{k}}^{\prime}x_{n_{k}})g(y_{n_{k}}^{\prime}y_{n_{k}%
})\longrightarrow0$, contradicting the choice of the subsequence. Hence
$f\otimes g\in\mathcal{L}$-\textrm{$wap($}$A\otimes_{\left\vert \pi\right\vert
}B)$.

\item Since $g_{/B^{2}}\neq0$, there exists $0\leq u,v\in B_{B}$ such that
$g(uv)\neq0$. Let $(x_{n})\subset B_{A}$ be a disjoint sequence and
$(y_{n})\subset$ $B_{A}$. Consider the sequences $(x_{n},u)$ and $(y_{n},v)$
in $A\times B$. Since $x_{n}\perp x_{m}$ for $n\neq m$, $(x_{n},u)$ is
positive sporadically disjoint. Hence, by the $L$-weakly almost periodicity of
$f\otimes g$,
\[
(f\otimes g)(y_{n}x_{n}\otimes vu)\longrightarrow0\text{.}%
\]
Thus $f(y_{n}x_{n})g(uv)\longrightarrow0$. It follows that $f(y_{n}%
x_{n})\longrightarrow0$, so $f\in\mathcal{L}$-$\mathrm{wap}(A)$.

\item Let $(a_{n}),$ $(a_{m}^{\prime})\subset B_{A}$ be such that the iterated
limits $\lim\limits_{m}\lim\limits_{n}f(a_{n}a_{m}^{\prime})=\alpha$ and
$\lim\limits_{n}\lim\limits_{m}f(a_{n}a_{m}^{\prime})=\beta$, exist. Choose
$u,v\in B$ with $g(uv)\neq0$. Define $u_{n}=a_{n}\otimes u$ and $v_{m}%
=a_{m}\otimes v$. Then $(u_{n})$ and $(v_{m})$ are bounded in $A\otimes
_{\left\vert \pi\right\vert }B$ . Consequently
\begin{align*}
\lim_{m}\lim_{n}(f\otimes g)(u_{n}v_{n})  &  =\alpha g(uv)\text{ and }\\
\lim_{n}\lim_{m}(f\otimes g)(u_{n}v_{m})  &  =\beta g(uv)\text{.}%
\end{align*}
Since $f\otimes g$ is wap, the two iterated limits coincide. Hence
$\alpha=\beta$, and therefore $f\in$\textrm{$wap$}$(A)$.
\end{enumerate}
\end{proof}

\begin{remark}
The assumption $g_{/B^{2}}\neq0$ in (2) and (3) is necessary. Indeed, if
$g_{/B^{2}}=0$ then $(f\otimes g)(uv)=0$ for all $u,v\in A\otimes_{\left\vert
\pi\right\vert }B$. Consequently, $f\otimes g$ is automatically both
$L$-weakly almost periodic and weakly almost periodic, for every $f\in
A^{\prime}$, regardless of whether $f$ itself belongs to $\mathcal{L}%
$-$\mathrm{wap}(A)$ or $\mathrm{wap}(A)$.
\end{remark}

The next result shows that, apart from the trivial case, Arens regularity of
the Fremlin projective tensor product forces regularity of each factor.

\begin{theorem}
\label{F}Let $A$ and $B$ be non trivial Banach lattice algebras.

\begin{enumerate}
\item If $\mathcal{L}$-\textrm{$wap$}$(A\otimes_{\left\vert \pi\right\vert
}B)=(A\otimes_{\left\vert \pi\right\vert }B)^{\prime}$ then $\mathcal{L}%
$-\textrm{$wap$}$(A)=A^{\prime}$ and $\mathcal{L}$-\textrm{$wap$%
}$(B)=B^{\prime}$.

\item If \textrm{$wap$}$(A\otimes_{\left\vert \pi\right\vert }B)=(A\otimes
_{\left\vert \pi\right\vert }B)^{\prime}$ then \textrm{$wap$}$(A)=A^{\prime}$
and \textrm{$wap$}$(B)=B^{\prime}$.
\end{enumerate}

In particular, if $A\otimes_{\left\vert \pi\right\vert }B$ is Arens regular,
then both $A$ and $B$ are Arens regular.
\end{theorem}

\begin{proof}
\ \ \ \ \ \ \ \ \ \ \ \ \ \ \ 

\begin{enumerate}
\item Assume that $\mathcal{L}$-\textrm{$wap$}$(A\otimes_{\left\vert
\pi\right\vert }B)=(A\otimes_{\left\vert \pi\right\vert }B)^{\prime}$. Let
$f\in A^{\prime}$. Since $B$ is nontrivial, $B^{2}\neq\left\{  0\right\}  $.
Hence, there exists $g\in B^{\prime}$ such that $g_{/B^{2}}\neq0$. By
proposition \ref{PPP}, applied to $f\otimes g$, we obtain $f\in\mathcal{L}%
$-\textrm{$wap$}$(A)$. A symmetric argument shows that $\mathcal{L}%
$-\textrm{$wap$}$(B)=B^{\prime}$.

\item The proof is identical.
\end{enumerate}

Finally, if $A\otimes_{\left\vert \pi\right\vert }B$ is Arens regular, then
\textrm{$wap$}$(A\otimes_{\left\vert \pi\right\vert }B)=(A\otimes_{\left\vert
\pi\right\vert }B)^{\prime}$. Therefore, by part (2) \textrm{$wap$%
}$(A)=A^{\prime}$ and \textrm{$wap$}$(B)=B^{\prime}$, so $A$ and $B$ are Arens regular.
\end{proof}

The following example presents two Arens regular Banach lattice algebras whose
Fremlin projective tensor product fails to be Arens regular.

\begin{example}
Consider the unitization of the Banach $f$-algebra $\ell^{1}$, that is
$A=\ell^{1}\oplus\mathbb{R}1$, where $1=(1,1,...)$, with the coordinatewise
order, multiplication
\[
(x+\alpha1)(y+\beta1)=xy+\alpha x+\beta y+\alpha\beta1.
\]
and norm $\left\Vert x+\alpha\right\Vert =\left\Vert x\right\Vert
_{1}+\left\vert \alpha\right\vert $. Then $A$ is a commutative unital Banach
lattice algebra. Since \textrm{$wap$}$(A)=$\textrm{$wap$}$(\ell^{1})\oplus
$\textrm{$wap$}$(\mathbb{R})=A$, $A$ is Arens regular. The Banach $f$-algebra
$c_{0}$ is also Arens regular. We now show that $A\otimes_{\left\vert
\pi\right\vert }c_{0}$ is not Arens regular.

Define the bilinear map $\varphi:A\times c_{0}\longrightarrow\mathbb{R}$ by
\[
\varphi(x+\alpha1,g)=\sum\limits_{k=1}^{\infty}x_{k}g_{k}\text{ for all
}x+\alpha1\in A\text{ and }g\in c_{0}\text{,}%
\]
since $\left\vert \varphi(x+\alpha,g)\right\vert \leq\left\Vert g\right\Vert
_{\infty}\left\Vert x\right\Vert _{1}\leq\left\Vert x+\alpha\right\Vert
\left\Vert g\right\Vert _{\infty}$, $\varphi$ induces a continuous linear
functional, still denoted by $\varphi$, on $A\otimes_{\left\vert
\pi\right\vert }c_{0}$, that is
\[
\varphi((x+\alpha1)\otimes g)=\sum\limits_{k=1}^{\infty}x_{k}g_{k}\text{ for
all }x+\alpha1\in A\text{ and }g\in c_{0}\text{,}%
\]
For $n\geq1$, define
\[
e_{n}=(0,0,..,0,\overset{n}{\overbrace{1}},0,..)\text{, }x_{n}=e_{n}+1\in
A\text{ and }g_{n}=(\overset{n\text{ terms}}{\overbrace{1,1,..1},}0,0..)\in
c_{0}.
\]
and set $u_{n}=\dfrac{x_{n}}{2}\otimes e_{n}\in A\otimes_{\left\vert
\pi\right\vert }c_{0}$ and $v_{m}=1\otimes g_{m}\in A\otimes_{\left\vert
\pi\right\vert }c_{0}$ for all $n,m\geq1$. Since $\left\Vert u_{n}\right\Vert
_{\left\vert \pi\right\vert }=\left\Vert \dfrac{x_{n}}{2}\right\Vert
\left\Vert e_{n}\right\Vert _{\infty}=1$ and $\left\Vert v_{m}\right\Vert
_{\left\vert \pi\right\vert }=\left\Vert 1\right\Vert \left\Vert
g_{m}\right\Vert _{\infty}=1$, the sequences $(u_{n})$ and $(v_{m})$ belong to
$B_{A\otimes_{\left\vert \pi\right\vert }c_{0}}$. Moreover,
\begin{align*}
\varphi(u_{n}v_{m}) &  =\varphi(\dfrac{x_{n}}{2}.1\otimes e_{n}g_{m}%
)=\varphi(\frac{1}{2}(e_{n}+1)\otimes e_{n}g_{m})\\
&  =\frac{1}{2}\sum\limits_{k=1}^{\infty}e_{n,k}^{2}g_{m,k}\\
&  =%
\genfrac{\{}{.}{0pt}{0}{0\text{ if }n>m}{\frac{1}{2}\text{ if }n\leq m}%
\end{align*}
Thus, we obtain
\[
\lim_{n}\lim_{m}\varphi(u_{n}v_{m})=\frac{1}{2}\text{ and }\lim_{m}\lim
_{n}\varphi(u_{n}v_{m})=0\text{.}%
\]
The iterated limits are different. Hence, $\varphi$ is not weakly almost
periodic, and consequently, $A\otimes_{\left\vert \pi\right\vert }c_{0}$ fails
to be Arens regular.
\end{example}

Recall that every linear functional $f\in(A\otimes_{\left\vert \pi\right\vert
}B)^{\prime}$ induces a regular operator $T_{f}:A\longmapsto B^{\prime}$
defined by
\[
T_{f}(x)(y)=f(x\otimes y),x\in A,y\in B.
\]
The mapping $T:f\longmapsto T_{f}$ is a surjective Banach lattice isometry
from $(A\otimes_{\left\vert \pi\right\vert }B)^{\prime}$ onto $\mathcal{L}%
^{r}(A,B^{\prime})$ (see \cite[1.I]{Fremlin}). We now give several conditions
ensuring the Arens regularity of the Fremlin projective tensor product.

\begin{theorem}
Let $A$ and $B$ be a Banach lattice algebras. Then the following hold.

\begin{enumerate}
\item If $A^{\prime}$ and $B^{\prime}$ have the PSP, then $A\otimes
_{\left\vert \pi\right\vert }B$ is Arens regular.

\item If $\mathcal{L}^{r}(A,B^{\prime})$ has an order continuous norm, then
$A\otimes_{\left\vert \pi\right\vert }B$ , $A$ and $B$ are Arens regular.

\item If $A$ and $B$ are Banach $f$-algebras, then $A\otimes_{\left\vert
\pi\right\vert }B$ is Arens regular.
\end{enumerate}
\end{theorem}

\begin{proof}
\ \ 

\begin{enumerate}
\item Assume that $A^{\prime}$ and $B^{\prime}$ have the PSP. By \cite[Theorem
1]{Pedro}, it follows that $\mathcal{L}^{r}(A,B^{\prime})=(A\otimes
_{\left\vert \pi\right\vert }B)^{\prime}$ has the PSP. Hence, by Theorem
\ref{Main}, $\mathcal{L}$-\textrm{$wap($}$A\otimes_{\left\vert \pi\right\vert
}B)=(A\otimes_{\left\vert \pi\right\vert }B)^{\prime}$ which implies that
$A\otimes_{\left\vert \pi\right\vert }B$ is Arens regular.

\item This follows from Theorem \ref{Main}.

\item If $A$ and $B$ are Banach $f$-algebras, then by \cite[Theorem
4.2]{Jaber}, the Fremlin projective tensor product $A\otimes_{\left\vert
\pi\right\vert }B$ is again a Banach $f$-algebra. Since every Banach
$f$-algebra is Arens regular, the result follows.
\end{enumerate}
\end{proof}

We conclude this paper with a final remark concerning the relationship between
the projective and the Fremlin projective tensor products of Banach lattice algebras.

\begin{remark}
Let $A$ and $B$ be Banach lattice algebras. Their projective tensor product
$A\otimes_{\pi}B$ is a Banach algebra. According to \cite[Remark 3.4]{Jaber},
there exists a unique continuous algebra homomorphism $\theta:A\otimes_{\pi
}B\longrightarrow A\otimes_{\left\vert \pi\right\vert }B$ which extends the
identity on the algebraic tensor product $A\otimes B$. A natural question
arises about the injectivity of $\theta$.

We show that, in general, the answer is negative. Indeed, let $K$ and
$K^{\prime}$ be compact Hausdorff spaces, each containing a perfect set (that
is, a closed subset without isolated points). Set $A=C(K)$ and $B=C(K^{\prime
})$. Since $A\otimes_{\left\vert \pi\right\vert }B$ is a Banach $f$-algebra,
it is Arens regular. Suppose, contrary to the claim, that $\theta$ is
injective. Then $A\otimes_{\pi}B$ would be isomorphic, as a Banach algebra, to
the subalgebra $\theta(A\otimes_{\pi}B)$ of the Arens regular algebra
$A\otimes_{\left\vert \pi\right\vert }B$. Consequently, $A\otimes_{\pi}B$
would be Arens regular. However, by \cite[Theorem 6.11]{Lij}, $C(K)\otimes
_{\pi}C(K^{\prime})$ is not Arens regular under the above assumptions on $K$
and $K^{\prime}$. This contradiction shows that $\theta$ cannot be injective.
\end{remark}

\end{document}